\documentclass[a4paper,12pt]{amsart}
\usepackage{amsmath,amssymb}
\usepackage[shortlabels]{enumitem}
\usepackage[cp1251]{inputenc}

\makeatletter
\@namedef{subjclassname@2020}{%
\textup{2020} Mathematics Subject Classification} \makeatother

\def\R{\mathbb R}
\def\H{\mathcal H}
\def\T{\mathcal T}

\def\l{\lambda}

\newtheorem{thm}{Theorem}
\newtheorem{prop}[thm]{Proposition}
\newtheorem{cor}[thm]{Corollary}

\title{On a geometric extremum problem for convex cones}

\author{Oleg Mushkarov}
\address{Institute of Mathematics and Informatics, Bulgarian Academy of Sciences, Acad. G. Bonchev 8,
1113 Sofia, Bulgaria}
\email{muskarov@math.bas.bg}

\author{Nikolai Nikolov}
\address{Institute of Mathematics and Informatics, Bulgarian Academy of Sciences, Acad. G. Bonchev 8,
1113 Sofia, Bulgaria
\vspace{1mm}
\newline Faculty of Information Sciences, State University of Library Studies and Information
Technologies, Shipchenski prohod 69A, 1574 Sofia, Bulgaria}
\email{nik@math.bas.bg}

\thanks{The second named author was partially supported by the Bulgarian National Science Fund,
Ministry of Education and Science of Bulgaria under contract
KP-06-N82/6.} \subjclass[2020]{51M16, 52A38, 49K21}

\begin{document}

\keywords{Philon line, convex cone, simplex}

\maketitle

\vspace{-8mm}
\begin{abstract} We discuss the optimization problem for minimizing the $(n-1)$-volume
of the intersection of a convex cone $K$ in $\R^n$ with a
hyperplane through a given point, first considered in \cite{We}.
We give a geometric characterization of the stationary hyperplanes
for this problem when $K$ is a hyperangle which partially answers
a question posed in \cite{We}. Moreover, we study the location of
the set $S$ of points for which there is a stationary hyperplane
as well as the infimum of the $(n-1)$-volumes of cone segments of
$K$ cut off by hyperplanes through a given boundary point of $K$.
As a model example we study in detail the non-negative orthant 
of $\R^n$. In this case $S$ is its interior and we show that every point 
of $S$ lies in a unique stationary hyperplane, which we describe in
terms of the unique real root of an irrational equation.
\end{abstract}

\section{Introduction}

Let $K$ be a closed convex cone in $\R^n$ $(n\ge 2)$ with vertex
at the origin $O$ and non-empty interior $K^o.$ Assume that $K$ is
pointed, that is $K\cap(-K)=\{O\}.$ In \cite{We}, the following
optimization problem is consider:

$(\ast)$ \emph{Minimize the $(n-1)$-volume of the intersection of
$K$ with a hyperplane $\H\not\ni O$ through a given point $A.$}

Let $(\cdot,\cdot)$ be the standard inner product in $\R^n.$
Write $\H$ as $\H(b)=\{x:(b,x)=1\}$ ($b\neq 0$) and denote by
$V(b)$ the $(n-1)$-volume of $K(b)=K\cap\H(b).$ Then $(\ast)$
means to minimize $V(b)$ under the constraint $(b,a)=1$
($a=\overrightarrow{OA}$). Note that $V(b)$ is finite exactly when
$K(b)$ is compact, i.e. $b$ is an inner point of the dual cone
$K^*$ of $K.$ Recall that $K^*=\{y:(y,x)\ge 0 , \forall x\in K\}.$
Since $K^*\cap(-K)=\{O\},$ we may assume $A\not\in(-K)$ and we may
consider only vectors $b\in K^+:=(K^*)^o=\{y:(y,x)>0 , \forall x\in K\}$
($K^+\neq\varnothing$ if and only if $K$ is pointed).
Call the respective hyperplanes admissible.

The main Result 3.1 in \cite{We} states that $b$ is stationary for
$V$ under the constraint $(b,a)=1,$ i.e. $D_b(V)=\l a$
($\l$ is the Lagrange multiplier), if and only if
\begin{equation}\label{main}
\overrightarrow{AH}=n\overrightarrow{AG},
\end{equation}
where $H$ is the orthogonal projection of $O$ on $\H(b)$ and $G$
is the centroid of $K(b).$ We call $\H(b)$ a stationary hyperplane
for $A.$

This result extends the classical case $n=2,$ where the respective
stationary line is called a Philon line (for more details, see e.g.
\cite{CC} and \cite{We}).

Note that \eqref{main} implies that any admissible hyperplane $\H$
is stationary exactly for one point from $\H.$ On the other hand,
denote by $S$ the set of points $A\in\R^n$ for which there is a
stationary hyperplane through $A.$ Since $(\ast)$ has a solution
for any $A\in K^o,$ we have that $K^o\subset
S\subset\R^n\setminus(-K).$ In \cite{We}, a general question is
posed, namely, to describe the set $S$ and to characterize the
respective stationary hyperplanes. In Proposition \ref{geom4} we
recall the answer to this question for $n=2$ and note that the
situation changes when $n\geq 3$. We also discuss the infimum
problem related to $(\ast)$ for a boundary point $A$ of $K$
(Proposition \ref{geom5}).

A main purpose of this note is to show that if $K$ is the
non-negative orthant $(\R^+_0)^n,$ then $S=K^o$ and any point from
$S$ lies in a unique stationary hyperplane (which solves
$(\ast)$). Moreover, we describe this hyperplane in terms of the
unique real root of an irrational equation (Proposition \ref{df}).
Applying this result (for $n=2$), we also find the length of the
shortest segment cut from this orthant $K$ by a line through a
given point $A\in K^o$ (Proposition \ref{dfl}).

\section{Geometric characterization of stationary hyperplanes}

In this section we give a geometric characterization of the
stationary hyperplanes when $K$ is a hyperangle, i.e. a closed
convex cone in $\R^n$ spanned by $n$ linearly independent vectors
$e_1, e_2,\dots, e_n$. Moreover, we study the location of the set
$S$ defined above as well as the infimum of the $(n-1)$-volumes of
cone segments of $K$ cut off by admissible hyperplanes through a
given boundary point of $K$.

Let $K$ be a hyperangle in $\R^n$ and $\H=\H(b)$ a hyperplane such
that $b\in K^+.$ This means that $\H$ intersects any ray
$Oe_i^\to,$ say at $A_i.$ Let $C$ be the circumcenter of the
$(n-1)$-simplex $\mathcal A =A_1A_2\dots A_n.$ For $n\ge 3,$
denote by $M$ the Monge point of $\mathcal A ,$ i.e. the common
point of the hyperplanes through the centroids of the
$(n-3)$-dimensional faces of $\mathcal A$ which are orthogonal to
the opposite edges of $\mathcal A$ (see e.g. \cite{EHM}). For
$A\in\H$ define $A'$ by
\begin{equation}\label{newp}
\overrightarrow{OA'}=
\frac{n-1}{2}\overrightarrow{OA}.
\end{equation}

\begin{prop}\label{geom1} For $n\ge 3,$
any two of the following three conditions
imply the third one:

(i) $\H$ is a stationary hyperplane for $A;$

(ii) $A'A_1=A'A_2=\dots=A'A_n;$

(iii) $H=M.$
\end{prop}

\begin{proof} Let $A''$ be the orthogonal projection of $A'$ on $\H$.
Then, by \eqref{newp},
\begin{equation}\label{ort}
\overrightarrow{A''A}=\frac{n-3}{2}\overrightarrow{AH}.
\end{equation}
On the other hand, it is a classical result (see e.g. \cite{EHM})
that
\begin{equation}\label{cmg}
\overrightarrow{CG}=\frac{n-2}{n}\overrightarrow{CM}.
\end{equation}
The proposition follows by \eqref{main}, \eqref{ort}, \eqref{cmg},
and the fact that (ii) is equivalent to $A''=C$.
\end{proof}

\noindent{\bf Remarks.} (i) If $n=2,$ (iii) is missing, i.e. $(i)\Leftrightarrow(ii);$
in other words, $l$ is a Philon line for $A$ exactly when $G$
is the midpoint of $[AH].$

(ii) Recall that a $k$-simplex $\Delta$ ($k\ge 2$) is called
orthocentric if the altitudes from the vertices to the opposite
facets have a common point, called orthocenter. It turns out that
this is equivalent to orthogonality of any edge to the respective
$(k-2)$-dimensional face. Note that the orthocenter of $\Delta$
coincides with its Monge point (see e.g. \cite{EHM}).

Assume now that $OA_1A_2\dots A_n$ is an orthocentric $n$-simplex
($n\ge 3$). Then $A_1A_2\dots A_n$ is an orthocentric
$(n-1)$-simplex and $H=M.$ Hence in this case
$(i)\Leftrightarrow(ii),$ too.

\begin{prop}\label{geom2} Let $K$ be a hyperangle.

(i) If $K^*\subset K,$ then $S=K^o$.

(ii) If an extreme ray of $K$ forms acute angles with the others, then
there exists a point $A\in\partial K$ with a unique hyperplane
solving problem $(\ast);$ in particular, $K^o\subsetneq S.$
\end{prop}

\begin{proof} Let $\H$ be an admissible hyperplane and $H$ be the orthogonal
projection of $O$ on $\H.$ For a point $X\in
H,$ denote by $(x_1,x_2,\dots,x_n)$ its barycentric coordinates
w.r.t. to the $(n-1)$-simplex $\mathcal A$ defined at the
beginning of this section.

(i) The condition $K^*\subset K$ is equivalent to $K^+\subset K^o$
which means that $H\in\mathcal A^o$ for any admissible
hyperplane $\H,$ i.e. $h_i>0$ for any $1\le i\le n.$

Let now $\H$ be a stationary hyperplane for $A.$ Then \eqref{main}
reads as
\begin{equation}\label{bar}
h_i+(n-1)a_i=1,\quad 1\le i\le n.
\end{equation}

Hence $a_i<1/(n-1)$ and then
$$a_k=1-\sum_{i\neq k}^na_i>1-(n-1).\frac{1}{n-1}=0.$$
This shows $A\in K^o$ and thus $S\subset K^o.$

(ii) We may assume $(e_1,e_i)>0$ for any $2\le i\le n.$ Denote by
$\mathcal A_1=A_1A_2\dots A_n$ the respective $(n-1)$-simplex to the hyperplane
$\H_1=\H(e_1).$ Consider the point $A\in\partial K$ with barycentric coordinates
$(0,\frac{1}{n-1},\dots,\frac{1}{n-1})$ w.r.t. $\mathcal A_1.$

Let $\H$ be an admissible hyperplane for $A$ and $\mathcal
A'=A_1'A_2'\dots A_n'$ be the orthogonal projection of the
respective $(n-1)$-simplex $\mathcal A$ on $H_1.$ Then $A_1'=A_1$ and
$A_i'$ lies on the ray $A_1A_i,$ $2\le i\le n.$ Denote by
$K_1$ the cone spanned by these $n-1$ rays.

Recall now (see e.g. \cite[Section 4, item (iv)]{We}) that if the
$(n-1)$-volume of the cone segment of $K_1$ cut off by an
$(n-2)$-plane through $A$ is minimal, then $A$ is the centroid of
the respective $(n-2)$-simplex. Note that $A_2\dots A_n$ is the
unique $(n-2)$-simplex with this property. So
$$V_{n-1}(\mathcal A)\ge V_{n-1}(\mathcal A')\ge V_{n-1}(\mathcal A_1)$$
and (ii) easily follows.
\end{proof}

It is easy to see that $K\subset K^*$ (resp. $K\subset K^+$) if
and only if the angles between the extreme rays of $K$ are non-obtuse
(resp. acute). Hence the condition $K\subset K^+$ implies that in
(ii) of Proposition \ref{geom2}. On the other hand, since
$(K^*)^*=K,$ we have that $K^*\subset K$ if and only if the angles
between the inward normals to the facets of $K$ are non-obtuse. In particular,
the non-negative orthant is self-dual, i.e. $K=K^*.$ Note also that in general
$K^o\cap K^+\neq\varnothing$ which follows by the hyperplane separation theorem.

Next, we claim that if $K$ is a closed trihedral angle in $\R^3$ with face angles
$\alpha\le\beta\le\gamma,$ then the law of cosines for $K$ shows that its dihedral
angles are non-acute, i.e. $K^*\subset K,$ if and only if
$\beta\ge\pi/2$ and $\cos\alpha\le\cos\beta\cos\gamma.$ In particular,
$\alpha\ge\pi/2$ implies $K^*\subset K.$ This observation remains true in any
dimension.

\begin{prop}\label{geom3} If the angles between the extreme rays of a hyperangle $K$
in $\R^n$ are non-acute, then $K^*\subset K.$
\end{prop}

\begin{proof} Suppose that $K^*\not\subset K.$ Then,
since $K\not\subset K^+$ and $K\cap K^+\neq\varnothing,$ there exists a
point $A\in K^+\cap\partial K.$ We may assume that $a=\sum_{i=2}^n\alpha_i e_i,$
where $\alpha_i\ge 0.$ Hence $0<(a,e_1)=\sum_{i=2}^n\alpha_i(e_i,e_1)\le 0,$ a contradiction.
\end{proof}

Now we will describe the set $S$ if $n=2.$ Set
$\alpha=\angle(e_1,e_2).$ Then $K^*\subset K$ if
$\alpha\ge 90^\circ$ and $K\subset K^+$ if $\alpha<90^\circ.$ So
Proposition \ref{geom1} implies that $S=K^o$ in the first case and
$K^o\subsetneq S$ in the second one. In fact, we know much more
(see e.g. \cite{Bo} and \cite{CC}).

If $\alpha<90^\circ,$ set
$$\theta=\arctan\left(\frac{1+\sin^2\alpha/2}{1+\cos^2\alpha/2}\right)^{3/2}$$
and consider the closed angle $T$ with vertex $O,$ the same
bisector as $K,$ and aperture $2\theta.$ Then $K\subsetneq T.$

\begin{prop}\label{geom4} (i) If $\alpha\ge 90^\circ,$ then $S=K^0$ and for any point from
$K^0$ there exists one stationary line (global minima point).

(ii) If $\alpha<90,$ then $S=T\setminus\{O\}.$ Moreover,

-- for $A\in K$ there exists one stationary line (global minima point);

-- for $A\in T^o\setminus K$ there exist two stationary lines (local minima and maxima points);

-- for $A\in\partial T\setminus\{O\}$ there exists one stationary line (saddle point).
\end{prop}

It is noted in \cite[Section 4 (ii)]{We} that the situation changes when $n\geq 3.$
For example, if $K$ is the closed trihedral angle spanned by the vectors
$e_1=(1,-1,-0.2), e_2=(1,1,-0.2),$ and $e_3=(1,0,0.01)$ ($K\subset K^+$),
then three stationary planes exist for the point $A=(1,0,0)\in K^o$
(two local minima points and one saddle point).
\smallskip

Finally, let us mention again that if $K$ is a closed convex cone
in $\R^n$ with vertex at the origin $O,$ then the problem $(\ast)$
makes sense if $K$ is pointed, $K^o\neq\varnothing,$ and
$A\not\in(-K).$ It has solution if $A\in K^o$ and it has no
solution if $A\not\in K\cup(-K),$ since in this case the infimum
$m_A$ of the $(n-1)$-volumes of cone segments of $K$ cut off by
admissible hyperplanes through $A$ is equal to $0.$ So, it is
natural to ask what happens if $A\in\partial K.$ The answer is given by the following
proposition which is more or less intuitively clear.

\begin{prop}\label{geom5} Let $A\in\partial K\setminus\{O\}.$

(i) If there is no hyperplane
$\T\ni O$ such that $A$ is a relatively inner point for
$K'=\partial K\cap \T,$ then $m_A=0.$

(ii) Otherwise, such a $\mathcal{T}$ is unique and $m_A>0.$ If $m_A$ is not attained, then
$m_A$ is equal to the minimal $(n-1)$-volume of cone segment $K''$ of $K'$ cut off by an
$(n-2)$-plane $\T'\subset\T$ through $A.$ Moreover,
the relation \eqref{main} remains true in this degenerate case.
\end{prop}

\begin{proof} For any $A\in\partial K\setminus\{O\}$ there is a maximal positive integer
$k\le n-1$ and a unique $k$-plane $\T\ni O$ such that $A$ is a
relatively inner point for the cone $K'=\partial K\cap\T.$ Let
$\T'\subset\T$ be a $(k-1)$-plane through $A$ that minimize the
$(k-1)$-volume of the cone segment $K''$ of $K'$ cut off by $\T'.$
Take $n-k$ extreme rays of $K$ not lying in $\T'$ and points on
them which approach $O.$ These points together with $\T'$
determine admissible hyperplanes $\H_i$ for $K$ and it is not
difficult to see that $V_{n-1}(K\cap\H_i)\to V_{n-1}(K'').$ Hence
$m_A=0$ for $k<n-1$ and $m\ge m_A$ for $k=n-1.$

Let $k=n-1.$ It turns out that $A$ is the centroid of $K'\cap\T'$ \cite[Section (iv)]{We}).
Then $\overrightarrow{AO}=n\overrightarrow{AG},$ where $G$ is the centroid of $K'',$
i.e. \eqref{main} holds ($H=O$).

It remains to show that if $k=n-1$ and $m_A$ is not attained, then $m_A\ge m.$ For this,
consider a sequence of admissible hyperplanes $\H_i\ni A$ such that $V_{n-1}(K\cap\T_i)\to m_A.$
\smallskip

\noindent{\it Claim.} $\T_i\to\T.$
\smallskip

\noindent{\it Proof of the claim.} Let
$\T_i=\T_i(b_i)=\{x:(b_i,x)=1\}.$ Since $m_A$ is not attained,
then $||b_i||\to\infty.$ On the other hand, $A\in (K')^o$ and
therefore there are  $n-1$ linearly independent vectors $e_1,
e_2,\dots, e_{n-1}\in K'$ such that $a=\sum_{k=1}^{n-1}\alpha_k
e_k,$ where $\alpha_k>0.$ Let $c_i=\frac{b_i}{||b_i||}.$ Since
$(c_i,e_k)>0$ and $(c_i,a)=\frac{1}{||b_i||}\to 0,$ it follows
that $(c_i,e_k)\to 0$ for any $1\le k\le n-1.$ This means that
$c_i$ tends to a normal vector of $\T,$ i.e. $\T_i\to\T.$\qed
\smallskip

Further, denote by $\Pi_i$ the orthogonal projection of $K\cap\T_i$ on $\T.$
Consider a point $O_i\in\Pi_i$ at minimal distance $r_i$ from $O.$ Let $K_i$
be the convex hull of $O_i$ and $B_i=K'\cap\T_i.$ Then
\begin{equation}\label{m1}
V_{n-1}(K\cap\T_i)\ge V_{n-1}(\Pi_i)\ge V_{n-1}(K_i)
\end{equation}
Set $h_i=\mbox{dist}(O_i,B_i),$ $d_i=\mbox{dist}(O,B_i),$ and note that
\begin{equation}\label{m2}\frac{V_{n-1}(K_i)}{V_{n-1}(K''_i)}=\frac{h_i}{d_i},
\end{equation}
where $K''_i$ is the cone segment of $K'$ cut off by the
$(n-2)$-plane $\T_i\cap\T .$ Now a similar argument to that in the
proof of the claim provides by contradiction a $d>0$ such that
$d_i\ge d.$\footnote{In fact, such an argument shows that if $A$
is an interior point of a closed convex pointed cone $L,$ then
there exists a $d_A>0$ such that $\mbox{dist}(O,\mathcal H)>0$ for
any admissible hyperplane $\H\ni A.$} On the other hand, $\T_i\to
\T$ easily implies that $r_i\to 0$ and hence
\begin{equation}\label{m3}
\left|\frac{h_i}{d_i}-1\right|\le\frac{r_i}{d_i}\to 0.
\end{equation}
Since $V_{n-1}(K\cap\T_i)\to m_A,$ then \eqref{m1}, \eqref{m2}, \eqref{m3}, and
$V_{n-1}(K''_i)\ge m$ imply that $m_A\ge m.$
\end{proof}

\section{The non-negative orthant}

Let $K$ be the non-negative orthant in $\R^n,$ i.e.
$K=(\R^+_0)^n.$ We already know that $S=K^o.$ For
$A=(a_1,a_2,\dots,a_n)\in K^o,$ set $b_i=\frac{n-1}{2}a_i,$ $1\le
i\le n,$ and
$$f_A(x)=\sum_{i=1}^n\frac{b_i}{b_i+\sqrt{b_i^2+x}}-\frac{n-1}{2}.$$
This function is continuous and strictly decreasing on the interval $[-\min b_i^2,+\infty),$
$f(0)=\frac{1}{2}$ and $f(+\infty)=\frac{1-n}{2}.$ Hence the equation $f_A(x)=0$
has a unique real root $\l_A,$ and $\l_A>0.$ Set $c_i=b_i+\sqrt{b_i^2+\l_A},$
$1\le i\le n.$

\begin{prop}\label{df} For any $A\in K^o$ there exists a unique stationary hyperplane.
For the respective $(n-1)$-simplex $\mathcal A$ we have that
$\overrightarrow{OA_i}=c_ie_i,$ $1\le i\le n,$ and
$$V_{n-1}(\mathcal A)=\frac{\prod_{i=1}^nc_i}{(n-1)!\sqrt{\l_A}}.$$
\end{prop}

\begin{proof} Let $\H$ be an admissible hyperplane. Then, by Remark (ii) after
Proposition \ref{geom1}, $\H$ is a stationary hyperplane for $A$
if and only if $A'A_1=A'A_2=\dots=A'A_n,$ where $A'$ is given by
\eqref{newp}. Let $\overrightarrow{OA_i}=x_ie_i,$ $1\le i\le
n.$ Then we get that
\begin{equation}\label{f1}
x_i^2-(n-1)x_ia_i=A'A_i^2-A'O^2=:\l,\quad 1\le i\le n,
\end{equation}
i.e.
$$1-(n-1)\frac{a_i}{x_i}=\frac{\l}{x_i^2}.$$
Summing up these equalities and using that
\begin{equation}\label{f2}
\sum_{i=1}^{n}\frac{a_i}{x_i}=1
\end{equation}
(since $A, A_1, A_2,\dots, A_n\in\H$), we get that
$\l\sum_{i=1}^nx_i^{-2}=1,$ i.e. $\l=d^2,$ where $d$ is the
distance from $O$ to $\H.$ By \eqref{f1},
$x_i=b_i+\sqrt{b_i^2+\l}$ and then \eqref{f2} implies that
$f_A(\l)=0.$ So $\l=\l_A,$ $x_i=c_i,$ and the formula for
$V_{n-1}(\mathcal A)$ follows.
\end{proof}

Let us note that the circumcenter $C$ of the simplex $\mathcal A$
coincides with $A$ if only if $n=3.$
\smallskip

Applying Proposition \ref{df} for $n=2$ gives the following known fact.

\begin{cor}\label{plane} The length of a shortest segment cut from the first quadrant in
$\R^2$ by a line through its interior point $A(a_1,a_2)$ is equal to
$(a_1^{2/3}+a_2^{2/3})^{3/2}$ and such a segment is unique.
\end{cor}

\begin{proof} Let $\l=(a_1a_2)^{2/3}(a_1^{2/3}+a_2^{2/3}).$ Since
$$b_1^2+\l=a_1^{2/3}(a_1^{2/3}/2+a_2^{2/3})^2,\quad
b_2^2+\l=a_2^{2/3}(a_2^{2/3}/2+a_1^{2/3})^2,$$
it easily follows that $\l=\l_A.$ Then
$c_1=a_1+a_1^{1/3}a_2^{2/3},$ $c_2=a_2+a_2^{1/3}a_1^{2/3},$ and hence
$V_1=(c_1^2+c_2^2)^{1/2}=(a_1^{2/3}+a_2^{2/3})^{3/2}.$
\end{proof}

Using the above corollary, we will generalize it in higher dimensions.

\begin{prop}\label{dfl} Let $A=A(a_1,a_2,\dots,a_n)$ with $0<a_1\le a_2\le\dots\le a_n.$
Then the length of a shortest segment cut from the non-negative orthant $K$ in $\R^n$
by a line through $A$ is equal to $(a_1^{2/3}+a_2^{2/3})^{3/2}.$
\end{prop}

\begin{proof} We may assume that $n\ge 3.$
We will first prove that the length of a shortest segment cut from
the set $Q=(\R_0^+)^2\times\R^{n-2}\supset K$ by a line through
$A$ is given by the above formula, such a segment is unique and it
lies in $K.$ To do this, set $L_1=\R_0^+\times\R^{n-1}$ and
$L_2=\R\times\R_0^+\times\R^{n-2}.$ Let a line through $A$ cuts
from $Q$ a segment $[A_1A_2],$ where  $A_1\in L_1$ and $A_2\in
L_2$. Denote by $\Pi$ the plane through $A$ which is orthogonal to
$\{0\}\times\{0\}\times\R^{n-2}.$ Let $A_1'$ and $A_2'$ be the
orthogonal projections of $A_1$ and $A_2$ on $\Pi.$ Then $A_1'\in
L_1,$ $A_2'\in L_2,$ $O\in [A_1'A_2'],$ and
$|A_1'A_2'|\leq|A_1A_2|$ with equality if and only if $A'_1=A_1$
and $A_2'=A_2.$ Hence the problem is reduced to that for the right
angle $Ol_1l_2\subset K,$ where $l_1=L_1\cap\Pi$ and
$l_2=L_2\cap\Pi.$ It remains to apply Corollary \ref{plane}.

Let now $[BC]$ be a segment cut from $K$ by a line through $A$. Setting
$L_k= \R^{k-1}\times\R_0^+\times\R_0^{n-k},$ $1\le k\le n,$ we may choose $i\neq j$
such that $B\in L_i$ and $C\in L_j.$ Then the same reasoning as above
shows that $$\mid BC\mid\geq
(a_i^\frac{1}{3}+a_j^\frac{1}{3})^\frac{3}{2}\geq
(a_1^\frac{1}{3}+a_2^\frac{1}{3})^\frac{3}{2}$$ and the
proposition is proved.
\end{proof}

This proof shows that for $n\ge 3$ the number of lines
through $A$ which cut from $K$ segments of minimal length is equal
to $k-1$ if $a_1<a_2=a_k<a_{k+1},$ and to $\binom{k}{2}$ if
$a_1=a_2=a_k<a_{k+1}.$ In particular, such a line is unique
if and only if $a_2<a_3.$

On the other hand, it is clear that among all lines $l\not\ni O$
through $A\not\in K^o$ there is no one which cuts from $K$ a
segment of minimal length.

\end{document}